\documentclass[a4paper, 11pt, twoside]{article}

\usepackage{amsmath, amscd, amsfonts, amssymb, amsthm, latexsym, url, color, todonotes, bm, framed, afterpage, rotating, pdflscape}
\usepackage{graphicx}
\usepackage[left=1.3in, right=1.2in, top=1in, bottom=1.3in, includefoot, headheight=13.6pt]{geometry}
\usepackage{mathtools}
\input{xy}
\xyoption{all}

\newtheorem{theorem}{Theorem}[section]

\newtheorem{proposition}[theorem]{Proposition}

\theoremstyle{definition}

\newtheorem{remark}[theorem]{Remark}
\newtheorem{example}[theorem]{Example}

\newcommand{\xysquare}[8]{
\[\xymatrix{
#1 \ar@{#5}[r] \ar@{#6}[d] & #2 \ar@{#7}[d]\\
#3 \ar@{#8}[r] & #4
}\]
}

\DeclareMathOperator*{\projlimf}{``\varprojlim''}
\DeclareMathOperator*{\holim}{\operatorname*{holim}}

\newcommand{\bb}{\mathbb}

\newcommand{\blob}{\bullet}

\newcommand{\comment}[1]{}

\newcommand{\ep}{\varepsilon}

\newcommand{\into}{\hookrightarrow}
\newcommand{\isoto}{\stackrel{\simeq}{\to}}

\newcommand{\op}{\operatorname}

\renewcommand{\phi}{\varphi}
\newcommand{\quis}{\stackrel{\sim}{\to}}
\newcommand{\res}{\overline}
\newcommand{\roi}{\mathcal{O}}

\newcommand{\sub}[1]{\mathrm{\scriptsize #1}}

\newcommand{\To}{\longrightarrow}
\newcommand{\ul}[1]{\underline{#1}}

\renewcommand{\cal}{\mathcal}
\renewcommand{\hat}{\widehat}
\renewcommand{\frak}{\mathfrak}

\renewcommand{\tilde}{\widetilde}

\renewcommand{\projlim}{\varprojlim}

\DeclareMathOperator{\dlog}{dlog}

\DeclareMathOperator{\Spa}{Spa}

\newcommand{\dotimes}{\otimes^{\bb L}}

\usepackage{fancyhdr}

\usepackage{sectsty}
\sectionfont{\Large\sc\centering}
\chapterfont{\large\sc\centering}
\chaptertitlefont{\LARGE\centering}
\partfont{\centering}

\begin{document}

\title{Integral $p$-adic Hodge theory -- announcement}

\author{B.~Bhatt, M.~Morrow, P.~Scholze}

\date{}

\maketitle

\section{Introduction}
Let $C$ be an algebraically closed, nonarchimedean field of mixed characteristic which is complete under a rank one valuation; let $\roi\subseteq C$ be its ring of integers, with maximal ideal $\frak m$ and residue field $k$. The reader may assume that $C$ is the completion of an algebraic closure of $\bb Q_p$.

The main aim of this note is to outline a proof of the following result, which was first announced by the third author during his series of Fall 2014 lectures at the MSRI. Details, generalisations, and further results will be presented in a forthcoming article. In particular, this will include a ``comparison isomorphism''-style result.

\begin{theorem}
Let $\frak X$ be a proper, smooth, formal scheme over $\roi$, and let $i\ge 0$. Assume that the crystalline cohomology $H^i_\sub{crys}(\frak X_k/W(k))$ of its special fibre $\frak X_k$ is torsion-free. Then the $p$-adic \'etale cohomology $H^i_\sub{\acute et}(X,\bb Z_p)$ of its generic fibre $X:=\frak X_C$ is also torsion-free.
\end{theorem}

There are previous results of a similar flavour, at least in the case where $\frak X$ is the $p$-adic completion of the base-change to $\roi$ of a proper smooth scheme  over a finite extension $V$ of $W(k)$, of ramification index $e$. Specifically, under stronger assumptions on torsion-freeness (for Hodge cohomology, and assuming degeneration of the mod $p$ Hodge-de Rham spectral sequence), G.~Faltings proves the result if $\dim X<p-1$, cf. \cite[Theorem 6]{Faltings1999}. On the other hand, X.~Caruso \cite{Caruso2008} proves that if $ie<p-1$, then the torsion in crystalline and in \'etale cohomology agree, which implies our theorem in that case. In the last section, we give an example showing that there may be torsion in crystalline cohomology while the \'etale cohomology is torsion-free; thus one cannot hope to extend Caruso's result to the general case.

As an example of our theorem which falls outside the scope of applicability of previous results, one may take for $\frak X$ an Enriques surface over a $2$-adic base. In that case, there is $2$-torsion in $H^2_\sub{\acute et}(X,\bb Z_2)$, so that our theorem implies the existence of $2$-torsion in $H^2_\sub{crys}(\frak X_k/W(k))$, and thus non-vanishing of $H^1_\sub{dR}(\frak X_k/k)$, which is a known pathology of Enriques surfaces in characteristic $2$, cf. \cite[Proposition 7.3.8]{Illusie1979}.

We note that the theorem may have applications to the study of torsion in the \'etale cohomology of Shimura varieties. For example, K.-W.~Lan and J.~Suh \cite{LanSuh2012} prove first a vanishing result for torsion in crystalline cohomology, and then deduce the same for \'etale cohomology using existing results in integral $p$-adic Hodge theory.

\section{Strategy of the proof}

Let $\smash{\roi^\flat:=\projlim_\phi\roi/p}$ be the tilt (in the terminology of \cite[\S3]{Scholze2012}; classically $\roi^\flat$ was denoted by $\cal R_\cal O$ \cite[1.2.2]{Fontaine1994}) of $\roi$, whose field of fractions $C^\flat$ is a perfectoid field of characteristic $p$, and let $\bb A_\sub{inf}:=W(\roi^\flat)$ be the first of Fontaine's period rings, on which we denote by $\phi$ the usual Witt vector Frobenius. We will require the elements $\ep:=(1,\zeta_p,\zeta_p,\dots)\in\roi^\flat$ and $\mu:=[\ep]-1\in \bb A_\sub{inf}$, where $\zeta_p,\zeta_{p^2},\dots\in\roi$ is a chosen compatible sequence of $p$-power roots of unity and $[\cdot]$ denotes the Teichm\"uller lifting. The results of our constructions will be independent of this chosen sequence of $p$-power roots of unity.

The theorem is proved by constructing a Zariski sheaf $\bb A\Omega_{\frak X/\roi}$ of complexes of $\bb A_\sub{inf}$-modules on any smooth, formal $\roi$-scheme $\frak X$, whose cohomology, when $\frak X$ is proper, interpolates the $p$-adic \'etale and crystalline cohomologies of the generic and special fibres of $\frak X$ respectively. More precisely, it satisfies the following local properties in general:
\begin{itemize}
\item[($\bb A$1)] $\bb A\Omega_{\frak X/\roi}\dotimes_{\bb A_\sub{inf}}\bb A_\sub{inf}[\tfrac1\mu]\simeq R\nu_*\bb A_{\sub{inf},X}\dotimes_{\bb A_\sub{inf}}\bb A_\sub{inf}[\tfrac1\mu]$, where $\nu$ is the projection from the pro-\'etale site of $X$ to the Zariski site of $\frak X$ (see \cite[\S3]{Scholze2013} for the necessary theory of the pro-\'etale site of a rigid analytic space);
\item[($\bb A$2)] $\bb A\Omega_{\frak X/\roi}\dotimes_{\bb A_\sub{inf}}W(k)\simeq W\Omega_{\frak X_k/k}^\blob$ after $p$-adic completion of the left-hand side;
\item[($\bb A$3)] $\bb A\Omega_{\frak X/\roi}$ is equipped with a $\phi$-semilinear endomorphism $\phi_{\bb A}$ which is compatible with the previous two isomorphisms.
\end{itemize}

Now assume that the hypotheses of the theorem are satisfied, i.e., that $\frak X$ is moreover proper and that the $i^\sub{th}$ crystalline cohomology of its special fibre is torsion-free. Then a Fitting ideal argument using ($\bb A$2) and ($\bb A$3) shows that the hypercohomology group $H^i_\bb A(\frak X):=H^i(\frak X,\bb A\Omega_{\frak X/\roi})$ is a finite free $\bb A_\sub{inf}$-module equipped with a $\phi$-semilinear endomorphism $\phi_\bb A$ which becomes an isomorphism after inverting $\phi(\xi)$, where $\xi:=\mu/\phi^{-1}(\mu)$, in the sense that $\phi_{\bb A}:\phi^*H^i_\bb A(\frak X)\otimes_{\bb A_\sub{inf}}\bb A_\sub{inf}[\tfrac1{\phi(\xi)}]\isoto H^i_\bb A(\frak X)\otimes_{\bb A_\sub{inf}}\bb A_\sub{inf}[\tfrac1{\phi(\xi)}]$; in other words, $H^i_\bb A(\frak X)$ is a Breuil--Kisin module in the sense of L.~Fargues \cite[\S4]{Fargues2015}.

Moreover, the Breuil--Kisin module $H^i_\bb A(\frak X)$ has the following \'etale and crystalline specialisations, which follow respectively from ($\bb A$1) and ($\bb A$2):
\begin{itemize}
\item $(H^i_\bb A(\frak X)\otimes_{\bb A_\sub{inf}} W(C^\flat))^{\phi_{\bb A}=1}\cong H^i_\sub{\acute et}(X,\bb Z_p)$;
\item there is a canonical inclusion $H^i_\bb A(\frak X)\otimes_{\bb A_\sub{inf}} W(k)\into H^i_\sub{crys}(\frak X_k/W(k))$ with torsion cokernel, which is an isomorphism if $H^{i+1}_\sub{crys}(\frak X_k/W(k))$ is also assumed to be torsion-free.
\end{itemize}
Hence $H^i_\sub{\acute et}(X,\bb Z_p)$ is a free $\bb Z_p$-module of the same rank as $H^i_\sub{crys}(\frak X_k/W(k))$. Assuming the existence of the complex $\bb A\Omega_{\frak X/\roi}$, this completes the sketched proof of the theorem and defines our new cohomology theory $H^i_\bb A(\frak X)$ interpolating the $p$-adic \'etale and crystalline cohomologies of the generic and special fibres of $\frak X$ respectively.

\begin{remark}\label{remark_algebrisable_case}
	When $\frak X$ is the $p$-adic completion of a proper smooth $\mathcal{O}$-scheme, the theorem follows more directly from the existence of $\bb A\Omega_{\frak X/\roi}$, at least if $H^{i+1}_\sub{crys}(\mathfrak{X}_k/W(k))$ is also torsion-free. Indeed, $R\Gamma_{\mathcal{O}^\flat}(\frak X) := R\Gamma(\frak X, \bb A\Omega_{\frak X/\roi}) \dotimes_{\mathbb{Z}} \bb F_p$ is a perfect complex of $\mathcal{O}^\flat$-modules, and perhaps the simplest new observable of our construction, whose fibres interpolate the \'etale and crystalline cohomology modulo $p$, thanks to ($\bb A$1) and ($\bb A$2):
	\begin{itemize}
		\item 	$R\Gamma_{\mathcal{O}^\flat}(\frak X) \dotimes_{\mathcal{O}^\flat} C^\flat \simeq  R\Gamma_\sub{\acute et}(X,\bb F_p) \dotimes_{\bb F_p} C^\flat$,
		\item $R\Gamma_{\mathcal{O}^\flat}(\frak X) \dotimes_{\mathcal{O}^\flat} k \simeq  R\Gamma(\frak X_k,\Omega^\blob_{\frak X_k/k})$.
	\end{itemize}

	By semicontinuity this gives $\dim_{\bb F_p}H^i_\sub{\acute et}(X,\bb F_p) \leq \dim_kH^i(\frak X_k, \Omega^\blob_{\frak X_k/k})$, which yields the third of the following (in)equalities:
	\[\hspace{-5mm} \dim_{W(k)[\tfrac1p]} H^i_\sub{crys}(\frak X_k/W(k))[\tfrac{1}{p}] = \dim_{\bb Q_p}H^i_\sub{\acute et}(X,\mathbb{Q}_p) \leq \dim_{\bb F_p}H^i_\sub{\acute et}(X,\bb F_p) \leq \dim_kH^i(\frak X_k, \Omega^\blob_{\frak X_k/k}). \]
	Here the first equality is a consequence of the crystalline comparison theorem and thus uses that $\frak X$ is algebraic, while the second inequality is formal. If we now assume further that $H^\ast_\sub{crys}(\frak X_k/W(k))$ for $\ast=i,i+1$ is torsion free, then the first and last dimensions in the above chain of (in)equalities are equal, and thus all the dimensions are equal. In particular, 
	\[\dim_{\bb Q_p}H^i_\sub{\acute et}(X,\mathbb{Q}_p) = \dim_{\bb F_p}H^i_\sub{\acute et}(X,\bb F_p),\]
	which implies that $H^i_\sub{\acute et}(X,\mathbb{Z}_p)$ is torsion free. (Note also that, by GAGA results \cite[Thm.~3.7.2]{Huber1996}, $H^i_\sub{\acute et}(X,\bb Z_p)$ identifies with the $p$-adic \'etale cohomology of the generic fibre of the proper smooth $\roi$-scheme of which $\frak X$ is the $p$-adic completion.)
\end{remark}

\begin{remark}
	One can lift $(\bb A$2) above to a slightly larger portion of $\bb A_{\inf}$: the base change $\bb A \Omega_{\frak X/\roi} \otimes^{\bb L}_{\bb A_{\inf}} \bb A_{\mathrm{crys}}$ can be identified with the (Zariski sheafified) crystalline cohomology of $\frak X$ relative to $\bb A_{\mathrm{crys}}$, where the latter is given the standard pd-structure.
\end{remark}

\section{Outline of construction of $\bb A\Omega_{\frak X/\roi}$}
It is sufficient to naturally (in an $\infty$-categorical sense) define the value of $\bb A\Omega_{\frak X/\roi}$ on arbitrary affine open formal subschemes of $\frak X$, so we may assume that $\frak X=\op{Spf}R$, where $R$ is the $p$-adic completion of a smooth $\roi$-algebra. Hence $X=\frak X_C$ is the rigid analytic space $\Spa(R[\tfrac1p],R)$ over $C$, on whose pro-\'etale site we denote by $\bb A_{\sub{inf},X}:=W(\projlim_\phi\roi_X^+/p)$ the first period sheaf from \cite[Def.~6.1]{Scholze2013}; note that $\bb A_{\sub{inf},X}$ is a sheaf of $\bb A_\sub{inf}$-modules.

\subsection{Modifying torsion via the d\'ecalage functor $L\eta$}\label{subsection_modifying_torsion}
We must first describe a general process for modifying torsion in complexes, based on P.~Deligne's \cite[1.3.3]{Deligne1971} d\'ecal\'e filtration and used previously by P.~Berthelot and A.~Ogus \cite[Def.~8.6]{BerthelotOgus1978}. Let $A$ be a ring, $f\in A$ a non-zero divisor, and $N$ a cochain complex of $A$-modules supported in non-negative degrees such that $H^0(N)$ is $f$-torsion-free. Then we may replace $N$ by a quasi-isomorphic complex -- still supported in non-negative degrees -- consisting of $f$-torsion-free $A$-modules, and we then define $L\eta_fN$ to be the subcomplex of $N$ given by $L\eta_fN^n:=\{x\in f^nN^n:dx\in f^{n+1}N^{n+1}\}$ for $n\ge 0$. Then $L\eta_f: D^{\ge 0}_{f-\sub{tf}}(A)\to D^{\ge 0}_{f-\sub{tf}}(A)$ is a well-defined functor, where $D^{\ge 0}_{f-\sub{tf}}(A)$ is the derived (or, more usefully, $\infty$-) category of complexes of $A$-modules supported in non-negative degrees with $f$-torsion-free $H^0$. Evidently there is a natural transformation $L\eta_f\to\op{id}$.

Among various easily established identities (e.g., if $g\in A$ is another non-zero divisor then $L\eta_{fg}=L\eta_fL\eta_g$ and $(L\eta_f-)\dotimes_AA/gA\simeq L\eta_{f\,\sub{mod}\,gA}(-\dotimes_AA/gA)$, under mild hypotheses), the most important is that there is a natural quasi-isomorphism \[L \eta_fN\dotimes_AA/fA\quis (H^\blob(N\dotimes_AA/fA),\text{ Bock}),\] where the right side denotes the complex formed from the cohomology of $N\dotimes_AA/fA$ with differential given by the Bockstein operator $H^n(N\dotimes_AA/fA)\to H^{n+1}(N\dotimes_AA/fA)$ associated to $f$.

\subsection{Constructing pro Witt complexes}
We now describe two general processes for building ``$F$-$V$-procomplexes over the $\roi$-algebra $R$'' in the sense of Langer--Zink \cite{LangerZink2004}, among which their relative de Rham--Witt complex $\{W_r\Omega_{R/\roi}^\blob\}_r$ is the initial object. Note first that Fontaine's usual homomorphism of $p$-adic Hodge theory $\theta:\bb A_\sub{inf}\to\roi$, whose kernel is generated by the non-zero divisor $\xi$, lifts to a homomorphism $\theta_r:\bb A_\sub{inf}\to W_r(\roi)$ whose kernel is generated by $\xi_r:=\mu/\phi^{-r}(\mu)$; the Restriction $R$, Frobenius $F$, and Verschiebung $V$ on the Witt vectors $W_r(\roi)$ correspond via the homomorphisms $\theta_r$ respectively to $\op{id}$, $\phi$, and $\xi\phi^{-1}$ on $\bb A_\sub{inf}$.

Let $D$ be a coconnective, commutative, differential graded (or, more generally, $\bb E_\infty$-) algebra over $\bb A_\sub{inf}$ which is equipped with a $\phi$-semilinear automorphism $\phi_D$ and such that $H^0(D)$ is $\mu$-torsion-free. Assume also that there are isomorphisms of $W_r(\roi)$-algebras $\lambda_r:W_r(R)\isoto H^0(D\dotimes_{\bb A_\sub{inf}}\bb A_\sub{inf}/\xi_r)$ for all $r\ge 1$ in such a way that the Restriction, Frobenius, and Verschiebung on the left correspond respectively to the canonical projection, $\phi$, and $\xi\phi^{-1}$ on the right.

\ul{First process:} Equipping the cohomology groups $\cal W_r^n(D)_\sub{pre}:=H^n(D\dotimes_{\bb A_\sub{inf}}\bb A_\sub{inf}/\xi_r)$ with the Bockstein differential $\cal W_r^n(D)_\sub{pre}\to \cal W_r^{n+1}(D)_\sub{pre}$ associated to $\xi_r$ makes $\cal W_r^\blob(D)_\sub{pre}$ into a commutative differential graded $W_r(\roi)$-algebra. Equipping them further with a Restriction, Frobenius, and Verschiebung given by ``$\phi^{-r}(\xi)^n$ times the canonical projection'', $\phi$, and $\xi\phi^{-1}$ makes $\{\cal W_r^\blob(D)_\sub{pre}\}_r$ into an $F$-$V$-procomplex over the $\roi$-algebra $R$. (To be precise, the necessary identity that $Fd\lambda_r([x])=\lambda_r([x^{p-1}])d\lambda_r([x])$ for all $x\in R$ does not appear to be automatic for arbitrary $D$, but will be satisfied in our cases of interest.) By the universal property of Langer--Zink's relative de Rham--Witt complex, there are then induced natural morphisms $\{W_r\Omega_{R/\roi}^\blob\}_r\to \{\cal W_r^\blob(D)_\sub{pre}\}_r$ of $F$-$V$-procomplexes over the $\roi$-algebra $R$.

\ul{Improved process:} It turns out that $\cal W_r^\blob(D)_\sub{pre}$ must be adjusted by some $\phi^{-r}(\mu)$-torsion. To do this, we modify the first process in order to equip the cohomology groups $\cal W_r^\blob(D):=H^\blob(L\eta_{\phi^{-r}(\mu)}D\dotimes_{\bb A_\sub{inf}}\bb A_\sub{inf}/\xi_r)$ with the structure of an $F$-$V$-procomplex over the $\roi$-algebra $R$. There are  natural resulting morphisms $\{W_r\Omega_{R/\roi}^\blob\}_r\to \{\cal W_r^\blob(D)\}_r\to \{\cal W_r^\blob(D)_\sub{pre}\}_r$ of $F$-$V$-procomplexes over the $\roi$-algebra $R$.

The following is the key technical step in the construction of $\bb A\Omega_{R/\roi}$:

\begin{theorem}\label{theorem_key_almost_iso}
Apply the ``Improved process'' to $D=R\Gamma_\sub{pro\acute et}(X,\bb A_{\sub{inf},X})$. Then the induced map $W_r\Omega_{R/\roi}^\blob\to \cal W_r^\blob(R\Gamma_\sub{pro\acute et}(X,\bb A_{\sub{inf},X}))$ descends to the $p$-adic completion of $W_r\Omega_{R/\roi}^\blob$, inducing an almost (wrt.~the ideal $W_r(\frak m)\subseteq W_r(\roi)$) quasi-isomorphism of complexes of $W_r(\roi)$-modules for each $r\ge1$: \[\hat{W_r\Omega_{R/\roi}^\blob}\stackrel{\sub{al.}}{\quis} \cal W_r^\blob(R\Gamma_\sub{pro\acute et}(X,\bb A_{\sub{inf},X}))\]
\end{theorem}
\begin{proof}
By a localisation and \'etale base change argument, the assertions reduce to the case that $R=\roi\langle\ul T^{\pm1}\rangle :=\roi\langle T_1^{\pm1},\dots,T_d^{\pm1}\rangle$. Let $\tilde R:=\roi\langle\ul T^{\pm1/p^\infty}\rangle$, which is  equipped with a continuous action of $\bb Z_p^d$ via $\roi$-algebra homomorphisms, where the generator in the $i^\sub{th}$-coordinate of $\bb Z_p^d$ acts via $T_i^{j/p^k}\mapsto\zeta_{p^k}^jT_i^{j/p^k}$, for $j\in\bb Z$ and $k\ge 0$. In fact, $\tilde R[\tfrac1p]$ is a perfectoid $C$-algebra occurring as the global sections of a pro-\'etale cover $\tilde X:=\projlimf_k \Spa(K\langle \ul T^{\pm1/p^k}\rangle, \roi\langle \ul T^{\pm1/p^k}\rangle)$ of $X$, and there is a corresponding Hochschild--Serre map $R\Gamma_\sub{cont}(\bb Z_p^d,H^0_\sub{pro\acute et}(\tilde X,\bb A_{\sub{inf},X}))\to R\Gamma_\sub{pro\acute et}(X,\bb A_{\sub{inf},X})$. After tensoring by $\bb A_\sub{inf}/\xi_r\cong W_r(\roi)$, this is an almost quasi-isomorphism (wrt.~the ideal $W_r(\frak m)\subseteq W_r(\roi)$) by Faltings' almost purity theorem \cite[Thm.~3.1]{Faltings1988} \cite[Thm.~6.5]{Scholze2013}.

Applying the ``Improved process'' also to $D=R\Gamma_\sub{cont}(\bb Z_p^d,H^0_\sub{pro\acute et}(\tilde X,\bb A_{\sub{inf},X}))$ we arrive at a commutative diagram of commutative differential graded $W_r(\roi)$-algebras
\[\xymatrix{
W_r\Omega_{R/\roi}^\blob\ar[rr]\ar[dr] && \cal W_r^\blob(R\Gamma_\sub{pro\acute et}(X,\bb A_{\sub{inf},X}))\\
&\cal W_r^\blob(R\Gamma_\sub{cont}(\bb Z_p^d,H^0_\sub{pro\acute et}(\tilde X,\bb A_{\sub{inf},X})))\ar[ur]&
}\]
where the right diagonal arrow is an almost quasi-isomorphism by the final assertion of the previous paragraph. Up to explicit $\phi^{-r}(\mu)$-torsion, the terms appearing in the complex at the bottom of the diagram are $H^\blob_\sub{cont}(\bb Z_p^d,W_r(\tilde R))$; the top left of the diagram is even more explicit, using Langer--Zink's description of their relative de Rham--Witt complex in the case of a polynomial algebra \cite[\S2]{LangerZink2004}. Using such explicit descriptions we verify directly that the diagonal left arrow becomes an {\em isomorphism} after $p$-adically completing $W_r\Omega_{R/\roi}^\blob$, and this completes the proof.
\end{proof}

\subsection{The final step: glueing}
From the definition of the improved $\cal W_r^\blob(-)$-process and the stated properties of the $L\eta$ functor, there is a chain of quasi-isomorphisms
\begin{align*}\cal W_r^\blob(R\Gamma_\sub{pro\acute et}(X,\bb A_{\sub{inf},X}))
&=\left(H^\blob(L\eta_{\phi^{-r}(\mu)}R\Gamma_\sub{pro\acute et}(X,\bb A_{\sub{inf},X})\dotimes_{\bb A_\sub{inf}}\bb A_\sub{inf}/\xi_r),\text{ Bock}\right)\\
&\simeq L\eta_{\xi_r}L\eta_{\phi^{-r}(\mu)}R\Gamma_\sub{pro\acute et}(X,\bb A_{\sub{inf},X})\dotimes_{\bb A_\sub{inf}}\bb A_\sub{inf}/\xi_r\\
&\simeq L\eta_\mu R\Gamma_\sub{pro\acute et}(X,\bb A_{\sub{inf},X})\dotimes_{\bb A_\sub{inf}}\bb A_\sub{inf}/\xi_r,
\end{align*}
and hence the previous theorem may be rewritten as an almost quasi-isomorphism $\hat{W_r\Omega_{R/\roi}^\blob}\simeq L\eta_\mu R\Gamma(X,\bb A_{\sub{inf},X})\dotimes_{\bb A_\sub{inf}}\bb A_\sub{inf}/\xi_r$.

Let $R_k = R\otimes_\roi k$. For each $r\ge 1$, we can reduce the preceding map from $\bb A_\sub{inf} / \xi_r\cong W_r(\roi)$ to $W_r(k)$ to obtain the lower horizontal map in the diagram
\[\xymatrix{
\bb A_r\Omega_{R/\roi}\ar[r]\ar[d] & L\eta_\mu R\Gamma_\sub{pro\acute et}(X,\bb A_{\sub{inf},X})\dotimes_{\bb A_\sub{inf}} \bb A_\sub{inf}/p^r \ar[d]\\
W_r\Omega_{R_k/k}^\blob\ar[r] & L\eta_\mu R\Gamma_\sub{pro\acute et}(X,\bb A_{\sub{inf},X})\dotimes_{\bb A_\sub{inf}} W_r(k),
}\]
defining the top left complex $\bb A_r\Omega_{R/\roi}$ of $\bb A_\sub{inf}/p^r$-modules via (homotopy) pullback. As the lower two terms are almost zero, the upper horizontal arrow is an almost isomorphism. Moreover, the right vertical map is a quasi-isomorphism after applying $-\dotimes_{\bb A_\sub{inf}/p^r} W_r(k)$, as $W_r(k)\dotimes_{\bb A_\sub{inf}/p^r} W_r(k)\simeq W_r(k)$. This implies that there is a quasi-isomorphism
\[
\bb A_r\Omega_{R/\roi}\dotimes_{\bb A_\sub{inf}/p^r} W_r(k)\simeq W_r\Omega_{R_k/k}^\blob.
\]
We may now finally define \[\bb A\Omega_{R/\roi}:=\holim_r \bb A_r\Omega_{R/\roi},\] which is quasi-isomorphic, up to $W(\frak m^\flat)$-torsion (we avoid the word ``almost'' here since $W(\frak m^\flat)\neq W(\frak m^\flat)^2$), to $L\eta_\mu R\Gamma(X,\bb A_{\sub{inf},X})$ and satisfies $\bb A\Omega_{R/\roi}\dotimes_{\bb A_\sub{inf}} W(k)\simeq W\Omega_{R_k/k}^\blob$ after $p$-adic completion of the left-hand side. The complex $\bb A\Omega_{R/\roi}$ is equipped with a $\phi$-semilinear operator $\phi_\bb A$ obtained by glueing those on $W\Omega_{R_k/k}^\blob$ and $L\eta_\mu R\Gamma_\sub{pro\acute et}(X,\bb A_{\sub{inf},X})$.

Since $\mu\in W(\frak m^\flat)$, we have
\[\bb A\Omega_{R/\roi}\dotimes_{\bb A_\sub{inf}}\bb A_\sub{inf}[\tfrac1\mu]
\simeq L\eta_\mu R\Gamma_\sub{pro\acute et}(X,\bb A_{\sub{inf},X})\dotimes_{\bb A_\sub{inf}}\bb A_\sub{inf}[\tfrac1\mu]
\simeq R\Gamma_\sub{pro\acute et}(X,\bb A_{\sub{inf},X})\dotimes_{\bb A_\sub{inf}}\bb A_\sub{inf}[\tfrac1\mu],\]
verifying ($\bb A$1). We have already verified ($\bb A$2). This completes the outline of the construction of $\bb A\Omega_{\frak X/\roi}$.

\subsection{q-de Rham complexes}
An alternative perspective on $\bb A\Omega_{\frak X/\roi}$ and its construction is offered by the idea of $q$-de Rham complexes.

If $A$ is a ring, $q\in A^\times$, and $U$ is a formal variable, then the ``$q$-de Rham complex'' $q\op-\Omega_{A[U^{\pm1}]/A}^\blob$ of the Laurent polynomial algebra $A[U^{\pm1}]$ is defined to be \[A[U^{\pm1}]\To A[U^{\pm1}]\dlog U,\quad U^j\mapsto [j]_qU^j\dlog U,\] where $\dlog U$ is a formal symbol and $[j]_q:=\tfrac{q^j-1}{q-1}$ is the ``$q$-analogue of the integer $j$''. In the more general case of a Laurent polynomial algebra in several variables $U_1,\dots,U_d$, set \[q\op-\Omega_{A[U_1^{\pm1},\dots,U_d^{\pm1}]/A}^\blob:=\bigotimes_{i=1}^mq\op-\Omega_{A[U_i^{\pm1}]/A}^\blob.\] Note that if $q=1$ then the $q$-de Rham complex equals the usual de Rham complex; more generally, it is a deformation of the de Rham complex over $\res A:=A/(q-1)$, in the sense that $q\op-\Omega^\blob_{A[U_1^{\pm 1},\dots,U_d^{\pm1}]/A}\otimes_A\res A= \Omega^\blob_{\res A[U_1^{\pm 1},\dots,U_d^{\pm1}]/\res A}$.

If $A$ is a topological ring, whose topology is $I$-adic for some ideal $I$, then define $q\op-\Omega^\blob_{A\langle U_1^{\pm 1},\dots,U_d^{\pm1}\rangle/A}$ to be the $I$-adic completion of $q\op-\Omega^\blob_{A[U_1^{\pm 1},\dots,U_d^{\pm1}]/A}$.

\begin{proposition}
If $R=\roi\langle T_1^{\pm1},\dots,T_d^{\pm1}\rangle$ then $\bb A\Omega_{R/\roi}$ is quasi-isomorphic to the $q$-de Rham complex $q\op-\Omega^\blob_{\bb A_\sub{inf}\langle U_1^{\pm 1},\dots,U_d^{\pm1}\rangle/\bb A_\sub{inf}}$ associated to $q=[\ep]\in\bb A_\sub{inf}^\times$.
\end{proposition}
\begin{proof}
It follows from the proof of Theorem \ref{theorem_key_almost_iso}, and the general definition of $\bb A\Omega_{R/\roi}$, that in this case $\bb A\Omega_{R/\roi}$ is quasi-isomorphic (and not just almost so) to \[L\eta_\mu R\Gamma_\sub{cont}(\bb Z_p^d,H^0_\sub{pro\acute et}(\tilde X,\bb A_{\sub{inf},X})).\] This complex can be explicitly calculated and is quasi-isomorphic to $q\op-\Omega^\blob_{\bb A_\sub{inf}\langle U_1^{\pm 1},\dots,U_d^{\pm1}\rangle/\bb A_\sub{inf}}$ ($U_i$ corresponds to the Teichm\"uller lift of $(T_i,T_i^{1/p},T_i^{1/p^2},\dots)\in \tilde R\,^\flat$).
\end{proof}

Our complex $\bb A\Omega_{\frak X/\roi}$ may therefore be interpreted as a natural extension of the $q$-de Rham complex $q\op-\Omega^\blob_{\bb A_\sub{inf}\langle U_1^{\pm 1},\dots,U_d^{\pm1}\rangle/\bb A_\sub{inf}}$, defined initially only for tori $\op{Spf}\roi\langle T_1^{\pm1},\dots,T_d^{\pm1}\rangle$ with a fixed choice of coordinates, to arbitrary smooth, formal $\roi$-schemes $\frak X$.

\section{Some examples}
We give an example illustrating the sharpness of our result, namely that it is possible to have more torsion in crystalline cohomology than in \'etale cohomology.

Although the following does not fall within the scope of our result, as it involves non-smooth stacks, it heuristically explains what can happen:

\begin{example}
	Let $G = \mathbb{Z}/p\bb Z$, and $H = \mu_p$, both viewed as finite flat group schemes over $\mathcal{O}_C$. Choose a map $\eta:G \to H$ that is an isomorphism on the generic fibre and trivial on the special fibre by using a primitive $p^\sub{th}$ root of unity in $C$. Let $T \to BH$ be the universal $H$-torsor and $S=T\times_{BH}BG$ its pullback along $\eta$. Then $S_C$ is a point as it is the universal $G$-torsor, while $S_k \cong BG_k \times \mu_p$ is non-reduced. It is then easy to see that the crystalline cohomology of $S_k$ has a lot of torsion, while the \'etale cohomology of $S_C$ is trivial.
\end{example}

One can easily push the preceding example into the world of proper smooth schemes:

\begin{example}
	Let $p = 2$, and fix an Enriques surface $S/\mathcal{O}$ whose special fibre is ``singular'' in the sense that it has fundamental group $\mathbb{Z}/2\bb Z$ \cite[\S3]{BombieriMumford1976}. Fix also an auxiliary elliptic curve $E/\mathcal{O}$, and choose a map $\mathbb{Z}/2\bb Z \to E$ that is non-trivial on the generic fibre and trivial on the special fibre. Pushing out the universal cover of $S$ along this map gives an $E$-torsor $D \to S$ whose generic fibre $D_C \to S_C$ is non-split, i.e., does not admit a section, but whose special fibre $D_k \to S_k$ is split.
	
Thus $D$ is a proper smooth $\mathcal{O}$-scheme of relative dimension $3$ and one can show, using the Leray spectral sequence for $D \to X$, that $H^2_\sub{\acute et}(D_C,\mathbb{Z}_2)$ is torsion free while $H^2_\sub{crys}(D_k/W(k))$ has non-trivial $2$-torsion.
\end{example}

\bibliographystyle{acm}
\bibliography{../../Bibliography}

\end{document}